\documentclass[12pt]{article}
\usepackage{srcltx}
\usepackage{amssymb,amsmath,amsfonts,amsthm}
\usepackage{cite}
\newcounter{parag}

\newtheorem{lem}{Lemma}
\newtheorem*{theorem}{Theorem}
\newtheorem*{defi}{Definition}
\newtheorem*{cor}{Corollary}
\newtheorem{prop}{Proposition}

\begin{document}
\begin{center}
{\bf \Large On a connection between the order of a finite group and the set of conjugacy classes size}

\medskip

{\bf I.B. Gorshkov}
\end{center}

{\it Abstract: Let $G\in\{p,q\}^*$ be a finite group with trivial center, where $p,q\in\pi(G)$ and $p>q>5$. In the present paper it is proved that $|G|_{\{p,q\}}=|G||_{\{p,q\}}$; in particular $C_G(g)\cap C_G(h)=1$ for every $p$-element $g$ and every $q$-element $h$.
\smallskip

 \smallskip
}

\section*{Introduction}
There has been considerable work over the years into the relation between the
structure of a finite group and the cardinality of its conjugacy classes. One of the first results in this direction, due to Burnside, is that if a finite group has a
conjugacy class with prime power cardinality the group is not simple. Many papers, for example \cite{vse} and \cite{navarro}, was studied the structure of groups in which the orders of some elements and the their conjugacy classes are known. This paper study the structure of groups with restrictions on the set of conjugacy classes size.

In this paper, all groups are finite.
The number of elements of a set $\pi $ is denoted by $|\pi|$. Denote the set of prime divisors of positive integer $n$ by $\pi(n)$, and the set $\pi(|G|)$ for a group $G$ by $\pi(G)$.
The greatest power of a prime $p$ dividing the natural number $n$ will be denoted by $n_p$.

Let $G$ be a group and take $a\in G $. We denote by $a^G$ the conjugacy class of $G$ containing $a$ put $N(G)=\{|x^G|, x \in G\} \setminus\{1 \} $. Denote by the $|G||_p $ number $p^n$ such that $N(G)$ contains $\alpha$ multiple of $p^n$ and avoids the multiple of $p^{n+1}$. For $\pi\subseteq\pi(G)$ put $|G||_{\pi}=\prod_{p\in \pi}|G||_p $. For brevity, write $|G||$ to mean $|G||_{\pi (G)}$. Observe that $|G||_p$ divides $|G|_p $ for each $ p\in\pi(G) $. However, $|G||_p $ can be less than $|G|_p $. 

\begin{defi}
Let $p$ and $q$ be distinct numbers. Say that a group $G$ satisfies the condition $\{p,q\}^* $ and write $G\in\{p, q\}^*$ if we have $\alpha_{\{p, q\}}\in\{|G||_p, |G||_q, |G||_{\{p, q\}}\}$ for every $\alpha \in N(G)$.
\end{defi}
In this paper we inspect the groups with $\{p,q\}^*$-properties.
\begin{theorem}
If $G\in\{p,q\}^*$ be a group with trivial center, where $p,q\in\pi(G)$ and $p>q>5$, then $|G|_{\{p,q\}}=|G||_{\{p, q\}}$.
\end{theorem}

\begin{cor}
In the hypotheses of the theorem, $C_G(g)\cap C_G(h)=1$ for every $p$-element $g$ and every $q$-element $h$.
\end{cor}

\section{Definitions and preliminary results}
\begin{lem}[{\rm \cite[Lemma 1.4]{GorA2}}]\label{factorKh}
For a finite group $G$, take $K\unlhd G$ and put $\overline{G}= G/K$. Take $x\in G$ and $\overline{x}=xK\in G/K$.
The following claims hold:

(i) $|x^K|$ and $|\overline{x}^{\overline{G}}|$ divide $|x^G|$.

(ii) For neighboring members $L$ and $M$ of a composition series of $G$, with $L<M$, take $x\in M$  and
the image $\widetilde{x}=xL$ of $x$. Then $|\widetilde{x}^S|$ divides $|x^G|$, where $S=M/L$.

(iii) If $y\in G$ with $xy=yx$ and $(|x|,|y|)=1$ then $C_G(xy)=C_G(x)\cap C_G(y)$.

(iv) If $(|x|, |K|) = 1$ then $C_{\overline{G}}(\overline{x}) = C_G(x)K/K$.

(v) $\overline{C_G(x)}\leq C_{\overline{G}}(\overline{x})$.
\end{lem}

\begin{lem}\cite[Theorem V.8.7]{Hup}\label{frgroup}
Let $G$ be a Frobenius group with kernel $A$ and complement $B$. Then
the following statements are true:

(i) $|B|$ divides $|A|-1$;

(ii) $A$ is nilpotent, and if the order of $B$ is even then $A$ is abelian;

(iii) the Sylow $p$-subgroups of the group $B$ are cyclic for odd $p$ and
cyclic or general quaternion groups for $p=2$.

(iv) Every subgroup of $B$ of order $pq$, where $p$ and $q$ are primes, is cyclic.
\end{lem}

\begin{lem}\cite[Theorem 2]{camina72}\label{camin2}
Let $G$ be a finite group whose conjugacy classes have either $1$, $p^a$, $q^b$
or $p^a q^b$ elements, where $p$ and $q$ are primes and $a$ and $b$ are integers. If all of the values
actually occur, then $G$ is nilpotent.
\end{lem}

\begin{lem}\label{hz2}
Take $g\in G$. If each conjugacy class of $G$ contains an element $h$ such that $g\in C_G(h)$ then $g\in Z(G)$.
\end{lem}
\begin{proof}
The assertion follows from the fact that a finite group generated by any set of representatives of conjugacy classes.
\end{proof}

The prime graph $GK(G)$ of $G$ is defined as follows. The vertex set is $\pi(G)$ and two distinct primes $r, s\in\pi(G)$ considered as vertices of the graph are adjacent if and only if there exists element $g\in(G)$ such that $|g|=rs$. Denote by $\pi_i(G)$ the set of vertices of the $i$-th prime graph component of $G$. If $G$ has even order
then we always assume that $2\in \pi_1(G)$.

\begin{lem}\cite[Theorem A]{Wil}\label{GK}
If a finite group $G$ has disconnected prime graph, then one of the
following conditions holds:
\begin{enumerate}
\item[(a)]{$s(G)=2$ and $G$ is a Frobenius or 2-Frobenius group;}
\item[(b)]{there is a non abelian simple group $S$ such that $S\leq G = G/F(G)\leq Aut(S)$, where $F(G)$ is the
maximal normal nilpotent subgroup of $G$; moreover, $F(G)$ and $G/S$ are $\pi_1(G)$-subgroups, $s(S)\geq s(G)$,
and for every $i$ with $2\leq i\leq s(G)$ there is $j$ with $2\leq j\leq s(S)$ such that $\pi_i(G)=\pi_j (S)$.}
\end{enumerate}
\end{lem}

\begin{lem}\cite[Theorem 1]{vse}\label{pat}
Let $G$ be a finite group, and let $p$ and $q$ be different primes. Then
some Sylow $p$-subgroup of $G$ commutes with some Sylow $q$-subgroup of $G$ if and only
if the class sizes of the $q$-elements of $G$ are not divisible by $p$ and the class sizes of
the $p$-elements of $G$ are not divisible by $q$.
\end{lem}
\begin{lem}\cite{navarro}\label{navarro}
Let $G$ be a finite group and $p$ a prime, $p\not\in\{3, 5\}$. Then $G$ has abelian Sylow $p$-subgroups if and only if $|x^G|_p=1$ for all $p$-elements $x$ of $G$.
\end{lem}

Given a set $\pi$ of primes, a finite group is said to have property $D_{\pi}$ whenever it includes a Hall $\pi$-subgroup and all its Hall $\pi$-subgroups are conjugate. For brevity, we write $G\in D_{\pi}$ when a group $G$ has property $D_{\pi}$.

\begin{lem}{\rm \cite{HallExB}}\label{HallEx}
Let $G$ be a group and $\pi$ be a set of prime numbers. If $G$ has a nilpotent
Hall $\pi$-subgroup then $G\in D_{\pi}$.
\end{lem}
\begin{lem}{\rm \cite[Corollary 6.7]{Revin}}\label{R}
Suppose that $G$ is a group and $\pi$ is a set of primes.
Then $G\in D_{\pi}$ if and only if all composition factor of $G$ have property $D_{\pi}$.
\end{lem}

\begin{lem}{\rm \cite[Theorem 3.3.2]{Gore}}\label{Gore332Mashke}
Let $G$ be a $p'$-group of automorphisms of an abelian $p$-group $V$ and suppose
$V_1$ is a $G$-invariant direct factor of $V$. Then $V=V_1\times V_2$, where $V_2$ is also
$G$-invariant.
\end{lem}

\begin{lem}{\rm \cite[Theorem 5.2.3]{Gore}}\label{Gore5hzOcomutantePstavtomor}
Let $A$ be a $\pi(G)'$-group of automorphisms of
an abelian group $G$. Then $G=C_G(A)\times[G,A]$.
\end{lem}




\begin{lem}\label{hz}
Let $S\leq G \leq Aut(S)$, where $S$ is a non abelian simple group such that $GK(G)$ has two connected components, $|\pi_2(G)|=1$ and for some $5<q\in \pi_1(G)$ the Sylow $q$-subgroups is abelian. Then there exists $g\in G$ such that $\pi(|g|)\subseteq\pi_1(G)$ and $|g^G|_q>1$.
\end{lem}
\begin{proof}
Suppose that $S$ is an alternating group. It is easy to show that $G$ contains a $2$-element $g$ such that $\pi(|C_G(g)|)\subseteq\{2,3\}$. Consequently, in this case the lemma is hold. Using \cite{GAP}, \cite{Wil}, and \cite{Kon}, we can infer that the lemma is true if $S$ is isomorphic one of the sporadic simple groups. Let $S$ be a group of Lie type over a field of characteristic $t$. It follows from \cite{Wil} and \cite{Kon} that $t\in\pi_1(G)$. Of \cite[Proposition 5.1.2]{Carter} shows that $S$ contains a regular unipotent element $h$, in particular, $\pi(|C_S(h)|=\{t\}$. If $t=q$ then $|C_G(x)|_t<|G|_t$ for every $q'$-element $x$. Since $|\pi_1(G)|>1$ we can take $g\in G$ with $\pi(|g|)\subseteq\pi_1(G)\setminus\{t\}$. Otherwise $g=h$.
\end{proof}

\section{Proof of the Theorem and the Corollary }


Assume that the theorem is false and that $|G|_p>|G||_p$. Then the centralizer of each element of $G$ contains an element of order $p$. Let $P\in Syl_p(G)$, $Q\in Syl_q(G)$.

\begin{lem}\label{qbolsheqchasti}
$|G|_q>|G||_q$.
\end{lem}
\begin{proof}
If $a\in Z(Q)$ then $|a^G|_q=1$ and hence $|a^G|_p=|G||_p$. Let $g\in C_G(a)$ be the $p$-element. If $|g^G|_p=1$ then $|g^G|_q=|G||_q$ and, consequently, the lemma is proved. Assume that $|g^G|_p=|G||_p$. Then $|(ag)^G|_p=|g^G|_p$. From Lemma \ref{factorKh} we have $C_G(ag)=C_G(a)\cap C_G(g)$ and hence $C_G(a)$ contains some Sylow $p$-subgroup of $C_G(g)$. Therefore, $C_G(a)$ include $Z(\overline{P})$ for some $ \overline{P}\in Syl_p(G)$. Thus $C_G(a)$ contains a $p$-element $g'$ such that $ |g'^G| _p = 1 $ and the lemma is proved.
\end{proof}

Lemma \ref{qbolsheqchasti} shows that $|C_G(a)|_q>1$ for every $a\in G$. In particular, the centralizer of every element contains an element of order $q$.


\begin{lem}\label{pqTh}
$\pi(G)\neq\{p,q\}$.
\end{lem}
\begin{proof}
Assume that $G$ is a $\{p, q\}$-group. We have $N(G)\subseteq \{|G||_p, |G||_q, |G||\} $. Suppose that $N(G)=\{|G||_p\}$ or $N(G)=\{|G||_p, |G||\}$. Take $g\in Z(P)$ where $P\in Syl_p$. Then $|g^G|_p=1$. Consequently, $|g^G|=1$ and $g\in Z(G)$; a contradiction. Assume that $N(G)=\{|G||_p, |G||_q\}$. Take $g\in Z(P)$, where $P\in Syl_p$. Lemma \ref{qbolsheqchasti} shows that  $C_G(g)$ contains a $q$-element $h$. We have $C_G(hg)=C_G(h)\cap C_G(g)$; in particular, $|(gh)^G|$ is a multiple of both $|g^G|$ and $|h^G|$. Therefore $C_G(h)=C_G(g)$. Now $C_G(h)$ contains an element $h'$ such that $|h'^G|_q=1$. Thus, $|(gh')^G|_{\{p,q\}}=|G||_p|G||q$; a contradiction. We have $N(G)=\{|G||_p, |G||_q, |G||\} $. It follows from the Lemma \ref{camin2} that $G$ is nilpotent. Hence, $Z(G)>1$; a contradiction.



\end{proof}

\begin{lem}\label{p}
For every $g\in P $ we have $|g^G|_{q}=|G||_q$.
\end{lem}
\begin{proof}
Assume that there exists $g\in P$ such that $|g^G|_q=1$. Let us show that in the centralizer of each element there is a subgroup conjugate to $Z(P)$. Take $x\in G$ such that $C_G(x)$ does not include the center of any Sylow $p$-subgroup of $G$.

Assume that $x=ab=ba$, where $a$ is a $p'$-element and $b$ is a non trivial $p$-element. Suppose that $Z(P)<C_G(a)$. If $|a^G|_p=1$ then $b^{C_G (a)}$ contain an element $c$ with $Z(P)<C_G(ac)$. Hence, $|a^G|_p=|G||_p $ and $|a^G|_p=|(ab)^G|_p $. Therefore, $x$ is contained in some Sylow $p$-subgroup of $C_G(a)$, and hence $C_G(x)$ includes the center of a Sylow $p$-subgroup. Thus, we may assume that $x$ is a $p'$-element.

Suppose that $x$ is a $\{p,q\}'$-element. If $|x^G|_p= 1 $ then clearly there exists $\overline{P}\in Syl_p(G)$ with $Z(\overline{P})<C_G(x) $; a contradiction. Hence, $|x^G|_p=|G||_p$. Take $h\in C_G(x)$ a $p$-element. If $|h^G|_p=|G||_p $ then $|(xh)^G|_p=|h^G|_p $. Hence, $C_G(x)$ includes a $p$-subgroup $S$ of $C_G(h)$. The group $S$ contains the center of some Sylow $p$-subgroup of $G$. Therefore, $|h^G|_p=1$ for every $p$-element $h$ of $C_G(x)$. As we noted above, the centralizer of each element contains an element of order $q$. Take $y\in C_G (x)$ a $q$-element. Assume that $|y^G|_p=|G||_p$. Then $|(xy)^G|_p=|y^G|_p$. Hence, $C_G(x)$ include some Sylow $p$-subgroup of $C_G(y)$. Since $C_G(g)$ includes a Sylow $q$-subgroup of $G$, we may assume that $g \in C_G(y) $. Hence in $C_G(x)$ there is a $p$-element $g'$ with $|g'^G|_p = |G||_p $; a contradiction. Thus, we may assume that $|y^G|_p=1$ for every $q$-element $y$ of $C_G(x)$. We have $|(xy)^G|_q=|y^G|_q=|G||_q$. Hence, $C_G(x)$ include some Sylow $q$-subgroup of $C_G(y)$. Every $q$-subgroup of $C_G(y)$ contains an element $z$ of the center of some Sylow $q$-subgroup of $G$. Therefore, $C_G(x)$ contains $z'$ such that $|z'^G|_q=1$; a contradiction.

Therefore, $x=ab=ba$ where $a$ is a $\{p,q\}'$-element and $b$ is a $q$-element. If $|a^G|_p=|G||_p$ or $|b^G|_p = |G||_p $ then the preceding arguments imply claim. Thus, $|a^G|_p=|b^G|_p=1$. Hence $|a^G|_q=|G||_q $. Since $C_G(a)$ include some Sylow $p$-subgroup of $G$, we may assume that $g\in C_G(a)$. Since $|(ag)^G|_q=|a^G|_q$, it follows that $C_G(ag)$ include some Sylow $q$-subgroup of $C_G(a)$. Hence $C_G(a)$ contains $h$ with $x\in C_G(g^h)$. We have $|(xg^h)^G|_p=|g^G|_p$ and hence $Z(P)<C_G(x^{h^{- 1}})$; a contradiction.

\end{proof}

\begin{lem}\label{q}
For every $g\in Q$ we have $|g^G|_{p}=|G||_p$.
\end{lem}
\begin{proof}
Similar to the proof of Lemma \ref{p}.
\end{proof}

\begin{lem}\label{pcentralizer}
We have $C_G(Q)=Z(Q)$ and $C_G(P)=Z(P)$.
\end{lem}
\begin{proof}
Assume that there is a $\{p, q\}'$-element of $a\in C_G (P)$. Take $c\in C_G (a)$ a $q$-element. Assume that $|c^G|_q=|G||_q $. We have $|(ac)^G|_q=|c^G|_q $ and, consequently, $C_G (a)$ contains a $q$-element $d$ centralizing some Sylow $q$-subgroup of $ G $. Therefore, $C_G(a)$ contains a $q$-element $b$ such that $ |b^G|_q=1 $. To infer that $b\in Z(G) $, it suffices to show that the centralizer of each element contains an element conjugate to $b$.

Assume that there exists $g\in G$ such that $C_G(g)\cap b^G=\varnothing$.
Let $g=xyz=xzy=zxy=zyx\in G$, where $x$ is a $q$-element, $y$ is a $p$-element and $z$ is a $\{p, q \}'$-element.

Assume that $g=x$. Since $b$ lies in the center of a Sylow $q$-subgroup of $G$, $C_G(g)$ contains an element conjugate to $b$; a contradiction.

Assume that $g=xy$ with non trivial $y$. We may assume that $y\in C_G(a) $. Lemma \ref{p} yields $|a^G|_q=|(ay)^G|_q=|y^G|_q$. Consequently, the Sylow $q$-subgroup of $|C_G(y)|$ is conjugate to a $q$-subgroup of the group $C_G(a)$. Since $|g^G|_q = |y^G|_q $, the Sylow $q$-subgroup of $C_G(g)$ is conjugate to the Sylow $q$-subgroup of $C_G(y)$. Hence, $C_G(g)$ contains an element conjugate to $b$; a contradiction.

Since there the centralizer of any element of $G$ contains both $p$- and $q$-elements, we may assume that $x$ and $y$ are not trivial and $g=xyz$. Note that $|g^G|_q\geq|z^G|_q$. Therefore, $C_G(g)$ contain some Sylow $q$-subgroup of $C_G(xy)$; in particular, $C_G(g)\cap b^G\neq\varnothing$.
Therefore, $Z(G)$ is non trivial; a contradiction with the assumption of the theorem. 
Thus, $C_G(P)=Z(P)$.
Similarly we show that $C_G(Q)=Z(Q)$.
\end{proof}

\begin{prop}\label{abeP}
For every $g\in P$ we have $|g^G|_p=1$. For every $h\in Q$ we have $|h^G|_q=1$.
\end{prop}
\begin{proof}



\begin{lem}\label{a1}
Assume that there is a $p$-element $g\in P$ with $|g|_p=|G||_p$. If $H\in Syl_p (C_G (g))$, then $H<Z(C_G (g))$.\end{lem}
\begin{proof}
Since the centralizer of each element contains an element of order $q$, we obtain $\pi(C_G(g))\neq\{p\}$.
Take a $p'$-element $x\in C_G (g)$. Then $|(gx)^G|_p=|g^G|_p$; therefore $x$ is contained in some Sylow $p$-subgroup $H$ of $C_G(g)$. It follows from Lemma \ref{pcentralizer}, that $|x^G|_p=|G||_p$.
Let $b\in H$. Then $|x^G|_p=|(xb)^G|_p$. Hence, $|b^{C_G (x)}|_p=1$. It follows from Lemma \ref{navarro} that $H$ is abelian. Thus, for each $h\in C_G(g)$ there exists $h'\in h^G$ such that $H<C_G(h')$. Lemma \ref{hz2} implies that $H\leq Z(C_G(g))$.
\end{proof}
\begin{lem}\label{a3}
Suppose that $g$ is the same as in Lemma \ref{a1} and $D\in Syl_q(C_G (g))$. The centralizer of each element of $G$ includes a subgroup conjugate to $D$.
\end{lem}
\begin{proof}
Take $h\in P$. If $h\in Z(P) $ then Lemma \ref{a1} yields $h\in Z(C_G (g))$ and hence $D\in C_G(h)$. We may assume that $|h^G|_p=|G||_p$. Lemma \ref{a1} implies that $C_G(h)\leq C_G(z)$, where $z \in Z(P)$. Since $|z^G|_q=|h^G|_q $ and $D^a<C_G(z)$ for some $a\in G$, it follows that $C_G(h)$ includes a subgroup conjugate to $D$. Thus, the centralizer of each $p$-element includes a subgroup conjugate to $D$. Take $x\in G$. Since there is the centralizer of each element contains a $p$-element, we may assume that $x=ab=ba$ where $a$ is a $p'$-element and $b$ is a nontrivial $p$-element. Lemma \ref{p} yields $|x^G|_q=|b^G|_q$. Hence, $C_G(x)$ includes a Sylow $q$-subgroup of $C_G(b)$. Since the $q$-subgroup of $C_G(b)$ is conjugate to $D$, it follows that $C_G(x)$ includes a subgroup conjugate to $D$.
\end{proof}

Lemmas \ref{a3} and \ref{hz2} shows that $D\leq Z(G)$. Since $D$ is non trivial, then $Z(G)>1 $; a contradiction. Thus, $|g^G|_p=1$ for every $g\in P$. For $h\in Q$ the proof is similar.
\end{proof}

Note that from Lemma \ref{navarro} it follows that the statement of Proposition \ref{abeP} is equivalent to the fact that the Sylow $p$- and $q$-subgroups of $G$ are abelian.


\begin{prop}\label{nachalo}
$C_G(P)$ or $C_G(Q)$ contain $\{p,q\}'$-element.
\end{prop}
\begin{proof}
Suppose that $C_G(P)\leq P$ and $C_G(Q)\leq Q$. Lemma \ref{pqTh} shows that $\{p,q\}\neq\pi(G)$.
Let $x\in G$ be a $\{p, q\}'$-element. Since $|x^G|_{\{p, q\}}=|G||_{\{p,q\}}$, we see that $|x^G|_{\{p,q\}}=|(xa)^G|_{\{p,q\}}$ for each $\{p,q\}$-element $a\in C_G(x)$. Hence, $|a^{C_G(x)}|_{\{p,q\}}=1 $. From Lemmas \ref{pat} and \ref{navarro}, a Hall $\{p,q\}$-subgroup of $C_G(x)$ exists and is abelian. By Lemma \ref{HallEx} all Hall $\{p,q\}$-subgroups of $C_G(x)$ are conjugate in $C_G(x)$. Take $\overline{P}\in Syl_p(C_G (x))$ and $\overline{Q}\in Syl_q(C_G(x))$ with $\overline{Q}\in C_G(\overline{P})$. Since $P$ and $Q$ are arbitrary we may assume that $\overline{P}<P$ and $\overline{Q}<Q$.

\begin{lem}\label{rc}
Let $C=C_G(\overline{P})$. Then $C$ is solvable, $C/\overline{P}$ is a Frobenius group or a double Frobenius group, $p$ is a connected component of the graph $GK(C/\overline{P})$. In particular, one of the groups $\overline{Q}$ and $P/\overline{P}$ is cyclic.
\end{lem}
\begin{proof}
Since $P$ is an abelian group and $|P|/|\overline{P}|=|G||_p$, it follows that $p\in\pi(C/\overline{P})$. Note that $p$ is a connected component of the prime graph of the group $\overline{C}=C/\overline{P}$. Otherwise, there exists a $p'$-element $y\in G$ such that $|C_G(y)|_p>|\overline{P}|$ and hence $|y^G|_p=1$; a contradiction with $C_G(P)\leq P$.

Since $x\in C$, it follows that $\pi(C)\neq\{p, q\}$. Let $z\in C$ be a $\{p,q\}'$-element. The Hall $\{p,q\}$-subgroup of $C_G(z)$ is abelian. Lemma \ref{HallEx} implies that all Hall $\{p,q\}$-subgroups of $C_G(z)$ are conjugate. Hence, $C_G(z)$ includes a Hall $\{p,q\}$-subgroup $X$ which is containing $\overline{P}$. Since $X$ is abelian, it follows that $X<C$. Thus, each $\{p,q\}'$-element of $C$ centralizes some Sylow $q$-subgroup of $C$. In particular, $GK(\overline{C})$ has two connected components.

Assume that $\overline{C}$ is not solvable. Lemma \ref{GK} shows $\overline{C}\simeq K.S.A$, where $K$ is a nilpotent $\pi_1(\overline{C})$-subgroup, $S$ is simple and $A\leq Out(S)$ is a $\pi_1(\overline{C})$-subgroup. Since $|\pi_1(\overline{C})|>1$, we have $\{p\}=\pi_2(\overline{G})$. Hence, $p$ does not divide $|K||A|$. Assume that $q\in\pi(K)$. Since $K$ is nilpotent, the Sylow $q$-subgrop $R$ of $K$ is unique and every Sylow $q$-subgroup of $\overline{C}$ includes $R$. Since the centralizer of each $p'$-element of $C$ contains some Sylow $q$-subgroup of $C$, it follows that all $p'$-elements of $\overline{C}$ centralize $R$. Since $S.A$ is generated by all its $p'$-elements, it follows that $R\leq Z(\overline{C})$; a contradiction with the property that $GK(\overline{C})$ is disconnected.

Therefore $|\overline{Q}|$ divides $|S.A|$. From Lemma \ref{hz}, if follows that $S.A$ contain a $p'$-element $b$ such that $|b^{S.A}|_q>1$; a contradiction with the property that the centralizer of each $p'$-element of $\overline{C}$ includes a Sylow $q$-subgroup of $\overline{C}$. Thus, $\overline{C}$ is a solvable Frobenius group or a double Frobenius group. By Lemma \ref{frgroup}, the Sylow $p$-subgroup or $q$-subgroup of $\overline{C}$ is cyclic.
\end{proof}

\begin{lem}\label{cent}
Take a $p$-element $h$ such that $C_G(h)$ contains a $\{p,q\}'$-element. Then there exists a subgroup $P_h$ of order $|G|_p/|G||_p$ such that $C_G(h)=C_G(P_h)$, in particular $P_h=Z(C_G (h))$.
\end{lem}
\begin{proof}
By definition, $C_G(h)$ contains a $\{p,q\}'$-element $x$. Take $P_h\in Syl_p(C_G(x))$ with $h\in P_h$. We said above that the Hall $\{p,q\}$-subgroup of $ C_G (x) $ is abelian. Obviously, $C_G(h)\geq C_G(P_h)$, and we verify the inverse inclusion. Take $Q_h\in Syl_q(C_G(x))$ with $Q_h\leq C_G (P_h) $. Since $|(hy)^{G}|_q=|h^{G}|_q $ for each $\{p,q\}'$-element $y\in C_G (h)$, it follows that there exists $a\in y^{C_G(h)}$ with $Q_h<C_{C_G (h)}(a)$. We have $|(ab)^{C_G(h)}|_p=|b^{C_G(h)}|_p=|a^{C_G (h)}|_p$, where $b\in Q_h $. Hence, $C_{C_G (h)}(a)$ includes some Sylow $p$-subgroup of $C_{C_G (h)}(b)$. Therefore $a^{C_G (h)}$ contains $a'$ such that $P_h Q_h<C_G(a')$. Hence, every $p'$-element of $C_G(h)$ centralizes some Sylow $p$-subgroup of $C_G(h)$. Since the Sylow $p$-subgroup of $G$ are abelian, every $p$-element of $C_G(h)$ centralizes some subgroup conjugate to $P_h$ in $C_G(h)$.
If $ba=ab=c\in C_G(h)$, $\pi(|a|)\subseteq\{p\}'$ and $\pi(b)=\{p\}$, then $|c^{C_G (h)}|_p=|a^{C_G (h)}|_p $. Hence, $C_G(c)$ includes a subgroup conjugate to $P_h$ in $C_G(h)$.

Hence, each conjugacy class of $C_G(h)$ contains an element of $C_G(P_h)$. Lemma \ref{hz2} implies that $C_G(h)\geq C_G(P_h)$. Thus, $P_h<Z(C_G (h))$.

Since $|C_G(y)|_p=|P_h|$ for every $p'$-element $y\in G$, it follows that $Z(C_G (h))=P_h$.
\end{proof}

\begin{lem}\label{cenq}
If $h$ is a $q$-element such that $C_G(h)$ contains a $\{p,q\}'$-element, then there exists a subgroup $\widehat{Q}_h$ of order $|G|_q/|G||_q$ with $C_G(h)=C_G(\widehat{Q}_h)$; in particular, $\widehat{Q}_h=Z(C_G (h))$. Moreover $C_G(h)/\widehat{Q}_h$ is a Frobenius group or double Frobenius group.
\end{lem}
\begin{proof}
Similar to the proof of Lemma \ref{cent}.
\end{proof}

\begin{lem}\label{piCgrrpq}
Take a $p$-element $g\in G$ with $\pi(C_G(g))=\{p,q\}$. For each $q$-element $h\in C_G(g)$ we have $\pi(C_G(h))=\{p,q\}$.
For each $a\in\langle g\rangle$ we have $\pi(C_G(a))=\{p,q\}$.
Take a $q$-element $g'\in G$ with $\pi(C_G(g'))=\{p,q\}$. For each $p$-element $h'\in C_G(g')$ we have $\pi(C_G(h'))=\{p,q\}$.
For each $b\in\langle g'\rangle$ we have $\pi(C_G(b))=\{p,q\}$.
\end{lem}
\begin{proof}
Assume that there exists a $\{p,q\}'$-element $x\in C_G(h)$. Lemma \ref{q} yields $|h^G|_p=|G||_p$. We have $|(xh)^G|_p=|h^G|_p$. Therefore, $x^G$ contains some $y\in C_G(g)$; a contradiction.
The proof of the assertion about $g'$ is similar.
We have $\langle g\rangle<C_G(h)$. Since $\pi(C_G(h))=\{p,q\}$, we obtain $\pi(C_G(a))=\{p,q\}$ for all $a\in\langle g\rangle$.
\end{proof}

Let $h$ be a $p$-element such that $\pi(C_G(h))\neq\{p,q\}$. Put $P_h=Z(C_G (h)$. Lemma \ref{cent} shows $|P_h|=|G|_p/|G||_p$. Since $Q_h$ is a Sylow $q$-subgroup of $C_G(h)$ and $P_h=Z(C_G (h))$, for each $p$-element $h$ the subgroup $P_h$ is uniquely determined. For all $a\in P_h$ we have $P_h=P_a$. Therefore, $P_h$ and $P_v$ either coincide or intersect trivially, where $v$ is the $p$-element whose centralizer contains the $\{p,q\}'$-element. Let $x \in C_G(h)$ be a $\{p,q\}'$-element. Therefore $P_h$ is a Sylow $p$-subgroup of $C_G(x)$. From Lemma \ref{rc} infer that $C_G(h)/P_h$ is a solvable Frobenius group or double Frobenius group.
The centralizer of each $p'$-element of $C_G(h)$ includes a group conjugate to $Q_h$. We have $p$ is a connected component of the graph $GK(C_G(h)/P_h)$. The centralizer of each $p'$-element of $C_G(h)$ includes a group conjugate to $Q_h$. In particular, by Lemma \ref{cenq} there exists $\widehat{Q}_{g'}$ for each $g'\in Q_h$ with $\widehat{Q}_{g'}=Z(C_G(g'))$. Since $Q_h<C_G(x)$, we conclude that $\widehat{Q}_{g'}=Q_h$. Since $P_h<C_G(Q_h)$, we find that $P_h\in Syl_p(C_G(g'))$.

\begin{lem}\label{abe}
The group $G$ contains a non-abelian composition factor whose order is a multiple of $p$ or $q$.
\end{lem}
\begin{proof}
Suppose that the orders of non-abelian composition factors of $G$ are not divisible by $p$ and $q$.
Consider $\widetilde{G}=G/O_{\{p, q\}'}(G)$ and the natural homomorphism $\widetilde{\ }:G\rightarrow \widetilde{G}$. Suppose that $\widetilde{P}\lhd\widetilde{G}$. Let $x\in G$ be a $p$-element such that $C_G(x)$ contains a $\{p, q\}'$-element. Since $\widetilde{Q_x}<C_G(\widetilde{P_x})$, it follows that $\widetilde{P}.\widetilde{Q_x}/\widetilde{P_x}$ is a Frobenius group with the kernel $\widetilde{P}/\widetilde{P_x}$ and the complement $\widetilde{Q_x}\widetilde{P_x}/\widetilde{P_x}$. Lemma \ref{frgroup} shows that $Q_x$ is a cyclic group. Let $H\in Hol_{\{p,q\}}(C_{\widetilde{G}}(\widetilde {Q_x}))$. Note that $ H/\widetilde{Q_x}$ is a Frobenius group and hence $Q/Q_x$ is a cyclic group. Take $b\in Q$ with $\langle bQ_x\rangle=Q/Q_x$ and assume that $b^z\in Q_x$ for some $z<|b|$. Lemma \ref{piCgrrpq} implies that $C_G(b)$ contains a $\{p,q\}'$-element. Therefore there exists $\widehat{Q}_b$. We have that for each $y\in Q_x$ there exists $\widehat{Q}_y$ and $\widehat{Q}_y=Q_x$. Therefore $Q_x=\widehat{Q}_b$; a contradiction. Thus, $Q=\langle a\rangle \times\langle b\rangle$, where $a$ satisfies $\langle a\rangle=Q_x$.

Assume that $O_{\{p,q\}}(G)$ is non trivial. The centralizer of each element of $O_{\{p,q\}}(G)$ contain a $p$-element. We may assume that $C=C_{O_{\{p,q\}}(G)}(x)>1$. It follows from Lemma \ref{cent} that $C=C_{O_{\{p,q\}}(G)}(P_x)$. The Sylow $p$-subgroup of $C_G(P_x)/P_x$ acts freely on $C/P_x$. Therefore $P/P_x$ is cyclic. Hence, $P\simeq P_x.\langle\overline{y}\rangle$. Take the preimage $y\in P$ of $\overline{y}$. Assume that $y^m \in P_x $, where $m<|y|$. If $\pi(C_G(y))=\{p,q\}$ then we arrive at a contradiction with Lemma \ref{piCgrrpq}. Therefore there exists $P_y$ and $P_y\neq P_y\cap P_x\neq 1$; a contradiction. Thus, $P=P_x\times\langle y \rangle $. We may assume that $y\in C_G(b)$. Thus, $C_G(\widetilde{ab})$ trivially intersects with $\widetilde{P}$. Since $\widetilde{P}$ is a normal subgroup of $\widetilde{G}$, it follows thet $|(ab)^G|_p=|P|$; a contradiction. Thus, the group $O_{\{p,q\}}(G)$ is trivial.

Observe that $C_G(Q_x)$ contains a $\{p,q\}'$-element. Lemma \ref{cenq} implies that $C_G(Q_x)$ is solvable and $q$ is a connected component of $GK(C_G(Q_x)/Q_x)$. Take $\overline{H}\in Hol_{p'}(C_G(Q_x)/Q_x)$. Since $Q=Q_x\times\langle b\rangle$, by the Schurr-Zassenhaus theorem $C_G(Q_x)$ contains a group $H$ isomorphic to $\overline{H}$. The number $q$ is the connected component of the graph $GK(H)$. It follows from Lemma \ref{GK} that $H$ is a Frobenius group or a double Frobenius group. Take $N\in Hol_{q'}(H)$. Since $P_x$ is the unique Sylow $p$-subgroup of $C_G(Q_x)$ and each $q'$-element of $C_G(Q_x)$ centralizes some Sylow $p$-subgroup of $C_G(Q_x)$, we have $N\in C_G(P_x)$. Hence, $N$ acts freely on $P/P_x$. Consequently, $N$ is cyclic. Therefore, $H$ is a Frobenius group. Since $C_{C_G(Q_x)}(P_x)\unlhd C_G(Q_x)$, we get $N\unlhd C_G (Q_x)$. Hence, $N$ is the Frobenius kernel of $H$.

Therefore $P.H/P_x$ is a double Frobenius group. Let $R<G$ be the minimal preimage of $P.H/P_x$. Since $N$ is cyclic, we infer that $R\cap P=[P,N]\times C_{R\cap P}(N)$. It follows from $C_{R\cap P}(N)<P_x$ that $C_{R\cap P}(N)$ is normal in $R$, and consequently  $C_{R\cap P}(N)$ is trivial. Since $P_x<C_P(N)$, we obtain $R\cap P_x=1$. Therefore, $R\simeq P.H/P_x$ and $P=[P,N]\times P_x$. For each $r\in[P,N]$, some subgroup of $C_R(r)$ is conjugate to $B\in Syl_q(R)$.

Let $l\in P\setminus P_x$ be such that $C_G(l)\cap Q\in Syl_q(C_G(l))$. Since $P=[P, N]\times P_x$, we can uniquely represent $l$ as $vg$, where $v\in[P,N]$ and $g\in P_x$. Assume that $g$ is non trivial. Take $w\in C_G(l)\cap Q$. Since $Q=Q_x\times\langle b\rangle$, we see that $w=cd$ with $c\in Q_x$ and $d\in \langle b\rangle$. We have $l=l^w=(vg)^{cd}=v^{cd}g^d$. The group $\langle b \rangle$ acts on $P_x$ without fixed points, and $[P, N]$ is a normal subgroup of $P.Q$. Consequently, $g^d\neq g$; which is a contradiction. Therefore, $l\in[P, N]$. Consequently, if $g$ is a $p$-element with $C_G(g)\cap Q \in Syl_q(C_G(g))$ then $g\in P_x$ or $g\in[P,N]$. In particular, for every $p$-element $h$ it follows that $h^G\cap P_x\neq\varnothing$ or $h^G\cap [P,N]\neq\varnothing$.

Assume that there exists $r\in P$ with $\pi(C_G(r))=\{p,q\}$. If $g\in r^G$ satisfies $C_G(g)\cap Q\in Syl_q(C_G(g))$, then $g\in[P,N]$. Hence, there exists $g'\in r^G$ such that $b\in C_G(g')$. Lemma \ref{piCgrrpq} yields that $\pi(C_G(b))=\{p,q\}$. Therefore, $\pi(C_G(h))=\{p,q\}$ for all $h\in[P,N]$.
Take a $p'$-element $f\in G$ and a $p$-element $s\in C_G(f)$. Since $\pi(C_G(s))\neq\{p,q\}$ we infer that $s^G\cap P_x\neq \varnothing$. We may assume that $s\in P_x$. Consequently, $f\in C_G(P_x)$. Since $Q<N_G(P_x)$, Lemma \ref{hz2} implies $P_x\lhd G$. By Lemmas \ref{Gore332Mashke} and \ref{Gore5hzOcomutantePstavtomor} there exists $X\lhd G$ with $P=X\times P_x$. Let $h=mn$ where $1\neq m\in P_x$ and $1\neq n\in X$. We have $h^G\cap P_x=\varnothing$ and $h^G\cap X=\varnothing$. We may assume that $C_G(h)\cap Q\in Syl_q(C_G(h))$. Since the centralizer of each element of $G$ contains an element of order $q$, we have $C_G(h)\cap Q>1$. If $y\in C_G(h)\cap Q$ then $y=cb^i$, where $c\in Q_x$. Therefore $h^y=m^{b^i}n^{cb^i}$. Since $X$ and $P_x$ are normal subgroup of $G$, we obtain $m^{b^i}\in P_x$ and $n^{ab^i}\in X$. Consequently, $m^{b^i}=m$. Since $\langle b\rangle$ acts freely on $P_x$, $b^i=1$ and $y=c$. Since $Q_x$ acts freely on $X$, we obtain $h^y\neq h$; a contradiction. This implies that $\pi(C_G(h))>\{p,q\}$ for all $h\in P$.

Take $c\in Q\setminus Q_x$. Then $C_P(c)<[P, N]$. Hence, $c$ is conjugate to some element of $\langle b\rangle$. Therefore, for any $c,d\in Q\setminus Q_x$ the groups $Q_c$ and $Q_d$ are conjugate. In a similar fashion, we can show that for all $c',d'\in Q\setminus Q_b$ the subgroups $Q_{c'}$ and $Q_{d'}$ are conjugate. Since the intersections $Q\setminus Q_x$ and $Q\setminus Q_b$ are not trivial, we infer that for all $c'',d''\in Q$, the subgroups $Q_{c''}$ and $Q_{d''}$ are conjugate. Since $P\lhd G$, the subgroups $P_{c''}$ and $P_{d''}$ are uniquely determined. Thus, for all $y\in P$, the subgroup $P_y$ is conjugate to $P_x$. Hence, each conjugacy class of $G$ contains an element of $C_G(P_x)$. Lemma \ref{hz2} yields that $P_x<Z(G)$; a contradiction.
Thus, $\widetilde{P}$ is not a normal subgroup of $\widetilde{G}$.

Similarly, we can show that $\widetilde{Q}$ is not a normal subgroup of $\widetilde{G}$.

Put $F=Fit(\widetilde{G})$ is a Fiting subgroup of $G$. The definition of $G$ and $\widetilde{G}$ implies that $F$ is a $\{p,q\}$-group. Since $p$ and $q$ do not divide the orders of non-abelian composition factors, $F$ is non trivial. Assume that $|F|$ is not divisible by $q$. Then $F$ is a $p$-group. Since the Sylow $p$-subgroup of $G$ is abelian, and any non-abelian composition factor of $G$ are not divisible by $p$, we get $\widetilde{P}=F$; a contradiction.

Therefore, $p,q\in\pi(F)$. Since $Q$ and $P$ are abelian, we obtain that $F$ contains a Hall $\{p,q\}$-subgroup of $C_G(g)$ for all $g\in F$. Consequently, $F=\widehat{P}\times\widehat{Q}$, where $\widehat{P}<P$ with $|\widehat{P}|=|P|/|G||_p$ and $\widehat{Q}<Q$ with $|\widehat{Q}|=|Q|/|G||_q$. The $\widetilde{P}/\widehat{P}$ acts freely on $\widehat{Q}$. Hence, $\widetilde{P}/\widehat{P}$ is cyclic. Take $h\in\widetilde{P}$ with $\langle h\widetilde{P}\rangle=\widetilde{P}/\widehat{P}$ and a $q$-element $s\in C_{\widetilde{G}}(h)$. Since $|C_{\widetilde{G}}(s)|_p=|P|/|G||_p$, we obtain $|h|\leq|P|/|G||_p $. Since $s\not\in\widehat{Q}$, we see that $s$ acts freely on $\widehat{P}$. Hence, $|\widetilde{G}/\widehat{P}|_p\geq|P|/|G||_p$. Consequently, $|\widetilde{G}/\widehat {P}|_p=|P|/|G||_p$ and $\langle h\rangle$ intersect trivially with $\widehat{P}$. Similarly we can show that $|\widetilde{G}/\widehat{Q}|_q=|Q|/|G||_q$ and $\widetilde{G}$ contains $f$, where $|f|=|Q|/|G||_q$ such that $f$ acting freely on $\widetilde{P}$. Thus, $|\langle\widehat{P},f\rangle|=|\langle\widehat{Q},h\rangle|=|G||_{\{p,q\}}$. We find that $\langle\widehat{Q},h\rangle$ is a Frobenius group with the kernel $\widehat{Q}$, $\langle\widehat{P},f\rangle$ is a Frobenius group with kernel $\widehat{P}$. This contradict the fact that the order of the kernel of the Frobenius group is greater than the order of the complement.
Thus, $G$ contains a non-abelian composition factor whose order is divisible by $p$ or $q$.

\end{proof}

Lemma \ref{abe} shows that $G$ contains a non-abelian composition factor whose order is divisible by $p$ or $q$. Assume that there exists a non-abelian composition factor of $G$ whose order is divisible by $p$.
Put $\overline{G}=G/O_{p'}$. Assume that $O_{p'}$ is not solvable. Then it is easy to show that the centralizer of any $p$-element is not solvable; a contradiction. Put $C=Soc(\overline{G})$. By definition, it follows that $|C|$ is a multiple of $p$. Assume that $C$ is not solvable. Then $C$ is a simple group; otherwise, the centralizer of some $p$-element is not solvable. We have $\overline{G}/C\leq Out(C)$. If $\overline{G}/C$ is divisible by $p$ then $\overline{G}$ contains an $p$-element $g$ that acting on $C$ as an outer automorphism. Since $p$ is greater than $5$, we infer that $g$ is a field or diagonal automorphism, and $C$ is a group of Lie type. Therefore $C_{\overline{G}}(g)$ is not solvable; a contradiction. Thus, $|C|_p=|G|_p $. Since $N(C)$ contains a number $\alpha$ with $\alpha_p=|C|_p$, we arrive at a contradiction.

Now $C$ is an elementary abelian $p$-group. Lemma \ref{abe} shows that $\overline{G}$ contains an non-abelian composition factor $\widetilde{S}$ whose order is a multiple of $p$. Let $S<\overline{G}$ be the minimal preimage of $\widetilde{S}$. Since $\widetilde{S}$ is a simple group, $S$ is generated by the set of $p$-elements. Since the Sylow $p$-subgroup of $\overline{G}$ is abelian and every Sylow $p$-subgroup of $\overline{G}$ includes $C$, we get $S<C_{\overline{G}}(C)$; this contradicts the property that the centralizer of each $p$-element of $\overline{G}$ is solvable.

\end{proof}

By Proposition \ref{nachalo}, there exists a $\{p,q\}'$-element $a\in C_G (P) $ in contradiction with Lemma \ref{pcentralizer}.







Thus, $|G|_{p, q}=|G||_{p, q}$. The proof of our theorem is complete.

Assume that $G$ contains a $p$-element $x$ and a $q$-element $y$ such that $C_G(x)\cap C_G(y)>1$. If $a\in C_G(x)\cap C_G (y)$ then $|a^G|_{\{p,q\}}\leq|G|_{\{p,q\}}/(pq)$ and hence $|a^G|_{\{p,q\}}=1$, which is a contradiction. This justifies the corollary.




\end{document}